\documentclass[10pt,%draft
]{amsart}

\usepackage{
amsmath, amsthm, amssymb, graphicx, commath, showkeys, pgfplots, pst-solides3d, inputenc, float,mathtools,scalerel,stackengine}
\newcommand{\no}[1]{#1}

\renewcommand{\no}[1]{}  \newcommand{\upDelta}{\Delta} 

\no{\usepackage{times}\usepackage[subscriptcorrection, slantedGreek,nofontinfo]{mtpro} \renewcommand{\Delta}{\upDelta}}

\date{\today}

\numberwithin{equation}{section}% 

\newtheorem{theorem}{Theorem}[section]
\newtheorem{proposition}{Proposition}[section]
\newtheorem{lemma}{Lemma}[section]
\newtheorem{definition}{Definition}[section]

\newtheorem{remark}{Remark}[section]

\theoremstyle{definition}

\DeclareMathOperator{\supp}{supp}

\DeclareMathOperator{\curl}{curl}

\newcommand{\R}{{\bf R}}
\newcommand{\C}{{\bf C}}

\newcommand{\A}{{\bf A}}

\newcommand{\be}[1]{\begin{equation}\label{#1}}
\newcommand{\ee}{\end{equation}}

\renewcommand{\d}{\mathrm{d}}
\renewcommand{\r}[1]{(\ref{#1})}
\renewcommand{\i}{\mathrm{i}}

\title[Inverse Scattering and Stability Estimates for The Biharmonic Operator]{Inverse Scattering and Stability Estimates for The Biharmonic Operator}

\author[S. RabieniaHaratbar]{Siamak RabieniaHaratbar}
\address{Department of Mathematics, Purdue University, West Lafayette, IN 47907}
\thanks{The author is partially supported by a NSF DMS grants No. 1600327 and 1900475}
\thanks{}

\begin{document}
\maketitle 
\vspace{-10mm}
\begin{abstract}
	We study an inverse scattering problem of a perturbed biharmonic operator. we show that the high-frequency asymptotic of scattering amplitude of the biharmonic operator uniquely determines $\curl \A$ and $V-\frac{1}{2}\nabla\cdot\A$. We also study the near-field scattering problem and show that the high-frequency asymptotic expansion up to certain error in terms of frequency $\lambda$ recovers the same two above quantities with no additional information about $\A$ and $V$. We also prove stability estimates for $\curl \A$ and $V-\frac{1}{2}\nabla\cdot\A$. 
\end{abstract}
 
\vspace{-2mm}
\section{Introduction}
	Consider the following biharmonic equation:
	\vspace{-1mm}
\be{1.1}
	(P - \lambda^4)u= (\Delta ^2 + \A\cdot\nabla + V -\lambda ^4)u=0. \vspace{-1mm}
\ee
	where $\Delta$ is Laplacian and $\cdot$ is the dot product $a\cdot b$ in $\R^n.$
	%=\sum^n _{k=1} a_kb_k$ for $a,b\in \R^n$. 
	Here $\A$ is a vector-valued function representing the magnetic field and $V$ is a scalar-valued function representing the potential function, with both $\A$ and $V$ regular enough and compactly supported. 

	The scattering and the inverse scattering problems for the Schr\"{o}dinger operators have long history, see [Lax-Reed]. One major application of the operator \r{1.1} is in the study of the theory of vibrations of beams and the elasticity theory, see \cite{1} for the case of the linear beam equation and \cite{7} for the nonlinear scattering problems. In a  work by Tyni and Serov \cite{15}, a Saito's type formula has been proved and it is shown that one can uniquely recover $V-\frac{1}{2}\nabla\cdot\A$, where $\A \in W_{p,2\delta} ^1(\R^n)$ and $V \in L^p _{2\delta}(\R^n)$. 
	
	Our goal in this paper is to find out what information about $\A$ and $V$ can be recovered from the high-frequency asymptotic of scattering amplitude. We do not consider the whole amplitude at zero or near zero frequencies (Calderon's problem). In the case of inverse boundary valued problem at zero frequency, one can fully reconstruct the potential $V$ and magnetic field $\A$, and in particular there is no gauge invariance (see \cite{4, 3, 5, 14}).
	The main result of this work is the following.
	\vspace{-1mm}
\begin{theorem}
	Let $V \text{and}\ \A \in C^k(\R^3)$ for $k$ large enough. Then the high-frequency asymptotic expansion of scattering amplitude $a(\omega,\theta,\lambda)$ up to $\mathcal{O}(\lambda^{-3})$ recovers $\curl \A$ and $V-\frac{1}{2}\nabla\cdot\A$ uniquely. 
\end{theorem}
	\vspace{-1mm}
	In other words, for another pair $(\tilde{\A},\tilde{V})$ with scattering amplitude $\tilde{a}$ so that $a=\tilde{a}+\mathcal{O}(\lambda^{-3}),$ then $\curl \A = \curl \tilde{\A}$ and $V-\frac{1}{2}\nabla\cdot\A = \tilde{V}-\frac{1}{2}\nabla\cdot\tilde{\A}.$  We prove the following high-frequency approximation of the scattering amplitude
	\vspace{-2mm}
$$
	a(\omega,\theta,\lambda)=\i\lambda\theta\cdot\hat{\A}(\lambda(\omega-\theta)) + \hat{V}(\lambda(\omega-\theta)) + \mathcal{O}(\lambda^{-1})
	\quad \text{\rm as \quad $\lambda \rightarrow \infty$}. \vspace{-1mm}
$$	 
	We also study the near-field scattering problem and show that knowing the high-frequency asymptotic expansion up to error of order $\lambda^{-3}$ recovers the same two above quantities but contains no additional information about $\A$ and $V$. Our recovery process  is constructive and explicit, and in principal stable, but we do not formally study stability.
	 
	For the well-studied scattering problem of the Schr\"{o}dinger equations, it has been shown that one can fully recover $\curl \A$ and $V$ and there is a gauge invariance, that is, for any two compactly supported magnetic fields with the same $\curl,$ the measurement cannot distinguish between them.   

	This paper is organized as follows: we state some preliminary definitions and results in section two. In section three, the asymptotic solution of the biharmonic equation has been formulated explicitly and appropriate error estimates have been established. Section four is devoted to near-field scattering problem and the proof of our main results is given in section five. The appendix is a collection of some technical formula and statements from \cite{15}.  
	
	\hspace{-4mm}\textbf{Acknowledgments.}
	The author would like to express his gratitude for Prof. {P. Stefanov} for suggesting this topic and many useful discussions throughout this work. The author thanks Prof. {V. Serov} and {T. Tyni} for carefully reading the manuscript and their comments which helped in improving the manuscript significantly.
	\vspace{-4mm}
\section{Preliminarily Results}
%	We first prove the following.
%\begin{lemma}
%	If $\A$ and $\tilde{\A}$ are two compactly supported magnetic fields with the same $\curl,$ then the measurement cannot distinguish between them.
%\end{lemma}	
%\begin{proof}
%	Let $\A$ and $\tilde{\A}$ be two magnetic field supported in a ball in $\R^n$. Then $\delta\A=\A-\tilde{\A}$ is a compactly supported function with $\curl(\delta\A)=0$. Since the ball in $\R^n$ is a simply connected domain, there exists a finite smooth function $\phi$ such that $\A=\tilde{\A}+\ d \phi$. In fact, let $x \in\supp (\delta\A)$ and $x_0$ be a point outside of $\supp (\delta\A)$. Let $c(t)$ be a path connecting $x_0$ to $x$. Define
%	$$
%	\psi(x) = \int_{x_0}^{x} (\delta\A)(c(t)) \cdot c^{\prime} (t)\ dt + \psi(x_0).
%	$$	
%	and set $\phi(x)=\psi(x)-\psi(x_0)$. Clearly $\phi$ is a smooth function with $\supp \phi = \supp (\delta\A)$. By the Stokes theorem, the above integral does not depend on the choice of $c(t)$ which implies that equality $\A = \tilde{\A} +\ d \phi$ holds. Note that outside of the support the potential function $\phi$ is just a constant $\phi(x_0)$. 
%\end{proof}
	Our goal is to find a special solution of equation \r{1.1} and corresponding amplitude which are called scattering solution and scattering amplitude. We first define the outgoing resolvent and solution which are fundamental notations. 
\begin{definition}
	We define the outgoing free resolvent operator $R_0(\lambda):=(\Delta^2 -\lambda^4)^{-1}$ from $C^{\infty} _0$ to $C^{\infty}$ as the analytic continuation of the operator
	\vspace{-2mm} 
$$
	\mathcal{F}(R_0(\lambda)f){\scriptstyle(\xi)}=\frac{\hat{f}(\xi)}{|\xi|^4 - \lambda^4}, \quad \text{from} \quad \mathcal{I}m\ \lambda > 0 \quad \text{to} \quad \C.\vspace{-1mm}
$$ 
\end{definition}
\begin{definition}
	Let $V \text{and}\ \A \in C^k(\R^3)$ for $k$ large enough. We denote the outgoing resolvent operator by $R(\lambda): C^{\infty} _0 (\R^n) \rightarrow C^{\infty}(\R^n)$ and define as
	\vspace{-1mm}
$$
	R(\lambda) = \lim_{\epsilon \rightarrow 0^{+}} (\Delta^2 + \A\cdot\nabla + V -\lambda^4 -\i\epsilon)^{-1}.\vspace{-1mm}
$$
\end{definition}
\begin{definition}
	Given $\lambda \in \C$, we say that the function $u$ is $\lambda$-outgoing, if there exists $c > 0$ and $f \in \mathcal{E}'$ such that $u_{|_{|x|>c}} = R_0(\lambda)f_{|_{|x|>c}}.$ In applications, the constant $c$ is larger than the radius of the support of the perturbations.
\end{definition}	
	For $f\in C^{\infty} _0(\R^3)$, we have the following integral operator representation
	\vspace{-1mm}
\be{2.1}
	[R_0(\lambda)f]{\scriptstyle(x)}=\int G_0(x,y,\lambda) f(y)\d y, \quad \text{ with \quad $G_0(x,y,\lambda)=\frac{e^{\i\lambda|x-y|}-e^{-\lambda|x-y|}}{8\pi\lambda^2|x-y|}$.}
\ee
	Note that $G_0$ is a 3-dimensional fundamental solution of $\Delta^2 -\lambda^4$, i.e. the kernel of $(\Delta^2 -\lambda^4-\i0)^{-1}$.

	The operator $\Delta^2 -\lambda^4$ can be written as $(-\Delta -\lambda^2)(-\Delta +\lambda^2)$. Since the operator $-\Delta +\lambda^2$ is elliptic, i.e. the principle symbol is $|\xi|^2 + \lambda^2$, there is no geometric optics. The operator $-\Delta -\lambda^2$ is the Helmholtz operator.

	We now formulate the scattering amplitude where the derivation mainly follows \cite{16}. Let $u=e^{\i\lambda x\cdot\theta}+u_{sc} $ be a solution for equation \r{1.1}, where $e^{\i\lambda x\cdot\theta}$ is the harmonic plane wave with incoming direction $\theta \in S^2$ (i.e. incident wave which is neither outgoing nor incoming), and $u_{sc}$ is the scattered solution which is assumed to be outgoing. To formulate the scattering amplitude $a$, we first formulate the the Lippmann-Schwinger integral equation. One has
	\vspace{-1mm}
$$
	(\Delta^2 + \A\cdot\nabla + V -\lambda^4)(e^{\i\lambda x\cdot\theta}+u_{sc})=0 
	\quad \Longrightarrow \quad (\Delta^2 -\lambda^4)(e^{\i\lambda x\cdot\theta}+u_{sc})=-(\A\cdot\nabla + V)u 
$$
	\vspace{-5mm}
$$
	\Longrightarrow \quad (\Delta^2 -\lambda^4)u_{sc}=-(\A\cdot\nabla + V)u \quad \Longrightarrow \quad u_{sc} = -[R_0(\lambda)(\A\cdot\nabla + V)]u.
$$
	Since $u_{sc}$ is an outgoing scattered solution, one may invert the operator $(\Delta^2 - \lambda^4)$ to have an explicit formula for the scattering solution $u_{sc}$ using equation \r{2.1}: 
	\vspace{-1mm}
$$
	u_{sc}= -\int G_0(x,y,\lambda) \big(\A{\scriptstyle(y)}\cdot\nabla + V{\scriptstyle(y)}\big) u{\scriptstyle(y,\theta,\lambda)}\d y. \vspace{-1mm}
$$
	Note that since $\A$ and $V$ are compactly supported, by definition $R_0(\lambda)(\A\cdot\nabla + V)$ is outgoing and therefore $u_{sc}$ is unique. On the other hand,
	\vspace{-1mm}
$$
	(\Delta^2 + \A\cdot\nabla + V -\lambda^4)(e^{\i\lambda x\cdot\theta}+u_{sc})=0 
	\quad \Longrightarrow \quad u_{sc} = -[R(\lambda)(\A\cdot\nabla + V)]e^{\i\lambda x\cdot\theta},
$$
	where $R(\lambda)$ is defined by Definition 2.2. We need to justify that the right hand side of the last equation is also outgoing. Assume that the resolvent exists for some $\lambda$. As it is shown in [Corollary 4.4, \cite{15}], by the resolvent identity, one has
	\vspace{-3mm}
$$
	R(\lambda) - R_0(\lambda) = -R_0(\lambda)(\A\cdot\nabla + V)R(\lambda) \quad \Longrightarrow \quad R(\lambda) = R_0(\lambda) + \sum^{\infty} _{k=1} \big(-R_0(\lambda)(\A\cdot\nabla + V)\big)^k R_0(\lambda).\vspace{-1mm}
$$	
	Using Agmon's estimates one can show the above series converges in $H^1 _{-\delta}$ and hence a unique solution exists. Therefore, for a compactly supported function, $R(\lambda)$ is a well-defined outgoing operator and one has the following important identity
	\vspace{-1mm}
$$
	u_{sc} = -[R_0(\lambda)(\A\cdot\nabla + V)]u = -[R(\lambda)(\A\cdot\nabla + V)]e^{\i\lambda x\cdot\theta}.
$$	  
	We are particularly interested in an outgoing solution of \r{1.1}. Since every outgoing solution has a far-field pattern (see \cite{11}), for any $(\omega,\theta,\lambda)\in S^{2}\times S^{2}\times \R^+$, there exists a function $a=a(\omega,\theta,\lambda)$ such that
	\vspace{-1mm}
$$
	u(x,\theta,\lambda)= e^{\i\lambda x\cdot \theta} - C_3\frac{e^{\i\lambda|x|}}{\lambda^2|x|}a(\omega,\theta,\lambda) + \mathcal{O}(\frac{1}{|x|^{2}}), \ \text{as $|x|\rightarrow \infty$}, \vspace{-1mm}
$$ 
	where $\omega=\frac{x}{|x|}$ is an outgoing direction. The scattering amplitude $a(\omega,\theta,\lambda)$ is given by
	\vspace{-2mm}
\be{2.2}
	a(\omega,\theta,\lambda)=\int e^{-\i\lambda\omega\cdot y}(\A{\scriptstyle(y)}\cdot\nabla + V{\scriptstyle(y)}) u{\scriptstyle(y,\theta,\lambda)}\d y.\vspace{-1mm}
\ee
	The scattering amplitude $a$ measures scattering in direction $\theta,$ for a plane wave at frequency $\lambda$ propagating in direction $\omega$. Next section provides necessary tools to proof the main result. We use the above representation of scattering solution to find appropriate estimates for our results. 
\vspace{-3mm}
\section{high-frequency Asymptotic Expansion of Biharmonic Solutions}
	%In this section we find 
	Consider the following ansatz expansion for the solution of biharmonic equation \r{1.1}:
	\vspace{-1mm}
$$
	u(x,\theta,\lambda)=e^{\i\lambda x\cdot \theta}(a_0+\frac{\i}{\lambda}a_1 + \frac{1}{\lambda^2}a_2 + \frac{1}{\lambda^3}a_3 + \frac{1}{\lambda^4}a_4 + \frac{1}{\lambda^5}a_5 + \mathcal{O}(\frac{1}{\lambda^6}))= e^{\i\lambda x\cdot \theta}\bold{a}.\vspace{-1mm}
$$
	Since the wave before entering the support is just a plane wave $e^{\i\lambda x\cdot\theta}$ propagating in direction $\theta \in S^{2}$, %after entering the support, 
	we assume the following initial condition where ${a_0(x,\theta)}_{|_{x\cdot\theta\ll0}}=1$ and ${a_i(x,\theta)}_{|_{x\cdot\theta\ll0}}=0$ for $i=1,2, 3, \dots .$ 
 
	The proposition below gives explicit expression for the coefficient $a_i$ for the measurement up to $\mathcal{O}(\lambda^{-3})$.
\begin{proposition}
	If $u$ is the solution of biharmonic equation \r{1.1}, then for $i=0, 1, \dots, 5$, the coefficient $a_i$ solves the following zero-initial condition system of equations: 
\be{3.1}	
	\left\{
	\begin{array}{ll}
		a_0=1, \quad a_1 = 0, \quad 4(\theta\cdot\nabla)a_2 = \theta\cdot\A,\quad
		4\i(\theta\cdot\nabla) a_3 = -2(\Delta+2(\theta\cdot\nabla)^2)a_2 + V,\\
		4\i(\theta\cdot\nabla) a_4 = -2(\Delta+2(\theta\cdot\nabla)^2)a_3 +\i(4(\theta\cdot\nabla)\Delta + \theta\cdot\A)a_2,\\
		4\i(\theta\cdot\nabla) a_5 = -2(\Delta+2(\theta\cdot\nabla)^2)a_4 +\i(4(\theta\cdot\nabla)\Delta + \theta\cdot\A)a_3 + \i(\Delta^2 + \A\cdot\nabla + V)a_2.
	\end{array}
	\right.
\ee
	Moreover, for
	\vspace{-2mm}
\be{3.2} 
	E(x,\theta,\lambda) := u(x,\theta,\lambda) - e^{\i\lambda x\cdot \theta}(1+ \frac{1}{\lambda^2}a_2 + \frac{1}{\lambda^3}a_3) = u(x,\theta,\lambda) - e^{\i\lambda x\cdot \theta}(1+ \bold{\tilde{a}}), \vspace{-1mm}
\ee
	the following estimates hold;
$$
	\norm {E(x,\theta,\lambda)}_{L^2(\R^3)} = \mathcal{O}(\lambda^{-4}), \quad \quad  \norm {E(x,\theta,\lambda)}_{H^1(\R^3)} = \mathcal{O}(\lambda^{-3}).\vspace{-1mm}
$$
\end{proposition}
\begin{proof}
	Let $u(x,\theta,\lambda)=e^{\i\lambda x\cdot \theta}\bold{a}$ be solution of biharmonic equation \r{1.1}. Since 
$$
	\left\{\begin{array}{ll}
	\A\cdot\nabla(e^{\i\lambda x\cdot \theta}\bold{a})= e^{\i\lambda x\cdot \theta}[\i \lambda \A\cdot\theta + \A\cdot\nabla]\bold{a}\\
	\Delta^2(e^{\i\lambda x\cdot \theta}\bold{a})=e^{\i\lambda x\cdot \theta}[\i \lambda\theta+\nabla]^4\bold{a}= e^{\i\lambda x\cdot \theta}[ \Delta^2 + 4\i\lambda(\theta\cdot\nabla)\Delta -2\lambda^2\Delta -4\lambda^2(\theta\cdot\nabla)^2 -4\i\lambda^3(\theta\cdot\nabla) +\lambda^4]\bold{a}
	\end{array}
	\right. 
$$
	we have 
	\vspace{-1mm}
$$
	(P - \lambda^4) u=e^{\i\lambda x\cdot \theta}[-4\i\lambda^3(\theta\cdot\nabla) -2\lambda^2 (\Delta+2(\theta\cdot\nabla)^2) + \i\lambda(4(\theta\cdot\nabla)\Delta + \A\cdot\theta) +  \Delta^2 + \A\cdot\nabla + V]\bold{a}= 0.\vspace{-1mm}
$$
	Rearranging all terms with respect to the power of $\lambda$ and equating singular coefficients, we get the following transport equations for $a_i$:
$$
	\left\{
	\begin{array}{ll}
	\mathcal{O}(\lambda^3): \ \ -4\i(\theta\cdot\nabla)a_0=0,\\
	\mathcal{O}(\lambda^2): \ \quad \ 4(\theta\cdot\nabla)a_1 -2(\Delta+2(\theta\cdot\nabla)^2)a_0=0,\\
	\mathcal{O}(\lambda): \ \quad -4\i(\theta\cdot\nabla)a_2 -2\i(\Delta+2(\theta\cdot\nabla)^2) a_1 + \i(4(\theta\cdot\nabla)\Delta + \theta\cdot\A)a_0=0,\\
	\mathcal{O}(1): \quad \ -4\i(\theta\cdot\nabla) a_3 -2(\Delta+2(\theta\cdot\nabla)^2)a_2 - (4(\theta\cdot\nabla)\Delta + \theta\cdot\A)a_1 + (\Delta^2 + \A\cdot\nabla + V)a_0=0,\\
	\mathcal{O}(\lambda^{-1}): \ -4\i(\theta\cdot\nabla) a_4 -2(\Delta+2(\theta\cdot\nabla)^2)a_3 +\i(4(\theta\cdot\nabla)\Delta + \theta\cdot\A)a_2 + \i(\Delta^2 + \A\cdot\nabla + V)a_1=0,\\
	\mathcal{O}(\lambda^{-2}): \ -4\i(\theta\cdot\nabla) a_5 -2(\Delta+2(\theta\cdot\nabla)^2)a_4 +\i(4(\theta\cdot\nabla)\Delta + \theta\cdot\A)a_3 + \i(\Delta^2 + \A\cdot\nabla + V)a_2=0.
	\end{array}
	\right. 
$$
	The first transport equation above and the initial condition ${a_0}_{|_{x\cdot\theta\ll0}}=1$ implies that ${a_0\equiv1}$. To compute $a_1$, by the second equation above we have
	\vspace{-1mm} 		
$$
	4(\theta\cdot\nabla)a_1(x) -2(\Delta+2(\theta\cdot\nabla)^2)a_0(x)=0 \quad    \Longrightarrow \quad  (\theta\cdot\nabla)a_1(x)= 0.\vspace{-1mm}
$$
	Since ${a_1}_{|_{x\cdot\theta\ll0}}=0$, therefore ${a_1\equiv0}$. Considering the transport equation corresponding $\mathcal{O}(\lambda)$, we have
	\vspace{-2mm}
$$
	-4\i(\theta\cdot\nabla)a_2(x) -2\i(\Delta+2(\theta\cdot\nabla)^2)a_1(x) + \i(4(\theta\cdot\nabla)\Delta + \theta\cdot\A(x))a_0(x)=0 \quad   \Longrightarrow  \quad (\theta\cdot\nabla)a_2(x) =\frac{1}{4}\theta\cdot\A(x).
$$
	Integrating the last equation along the flow $x+t\theta$ yields
	\vspace{-2mm}
$$
	\int_{-\infty}^0 (\theta\cdot\nabla)a_2(x+s\theta) \ ds =\frac{1}{4} \int_{-\infty}^0 \theta\cdot\A(x+s\theta)\ ds .
$$	
	Since $\theta\cdot\nabla=\partial_t$ along the null bi-characteristics (i.e. $x\cdot\theta=t$), therefore
	\vspace{-2mm}
$$
	\int_{-\infty}^0 \partial_s a_2(x+s\theta) \ ds =\frac{1}{4} \int_{-\infty}^0 \theta\cdot\A(x+s\theta)\ ds \quad  \Longrightarrow   \quad a_2(x) =\frac{1}{4} \int_{-\infty}^0 \theta\cdot\A(x+s\theta)\ ds,
$$ 
	which is the $X$-ray transform of the magnetic field $\A$ along the lines $x\cdot\theta=t$. Notice that for $i=3, 4, 5$ all the coefficients $a_i$ depend on the potential and magnetic fields $V$ and $\A$ can be computed recursively, by considering the transport equation corresponding to $\mathcal{O}(\lambda^{-i})$ and integrating along the flow as above. 

 	Our next goal is to establish an estimate for the error term $E(x,\theta,\lambda)$ in \r{3.2}. Since $u$ solves the biharmonic equation, we have
 	\vspace{-2mm}
$$
	-(P-\lambda^4)E(x,\theta,\lambda) = (P-\lambda^4)\big(e^{\i\lambda x\cdot \theta}(1 + \bold{\tilde{a}})\big) 
$$
$$
	=e^{\i\lambda x\cdot \theta}[-4\i\lambda^3(\theta\cdot\nabla) -2\lambda^2 (\Delta+2(\theta\cdot\nabla)^2) + \i\lambda(4(\theta\cdot\nabla)\Delta + \A\cdot\theta) +  \Delta^2 + \A\cdot\nabla + V]\big(1+ \bold{\tilde{a}} \big).
$$
	Expanding the right hand side of above equation and using transport equations \r{3.1} implies that
	\vspace{-2mm}
$$
	-(P-\lambda^4)E(x,\theta,\lambda) =\frac{e^{\i\lambda x\cdot \theta}}{\lambda}\big(- 2(\Delta+2(\theta\cdot\nabla)^2)a_3 +\i\lambda^2(4(\theta\cdot\nabla)\Delta + \A\cdot\theta)(\bold{\tilde{a}}) + \lambda (\Delta^2+\A\cdot\nabla+V)(\bold{\tilde{a}})\big).\vspace{-1mm}
$$
	Notice that although the scattering solution $u_{sc}=u-e^{\i\lambda\cdot\theta}$ is outgoing (see Definition 2.2), the above error term $E$ is not outgoing as $E$ has infinite support. To apply the resolvent $R(\lambda)$, we first need to localize the right hand side of above error in $L^2(\R^3)$: Let the compact set $K$ denotes the support of perturbation $\nabla\cdot\A + V$, and $\chi\in C^\infty _0(\R^3)$ be a smooth cut-off function such that $\chi(x)=1$ near the support $K$. We define
	\vspace{-3mm} 
$$
	E_{\chi}(x,\theta,\lambda):= u(x,\theta,\lambda)-e^{\i\lambda x\cdot \theta}(1+\chi\bold{\tilde{a}}),
$$
	where $u=e^{\i\lambda x\cdot \theta}+u_{sc}$ solves the biharmonic equation \r{1.1}. For all $x$ away from $K$, $\chi(x)=0$, and therefore $E=u_{sc}$ is outgoing. On the other hand, for $\lambda \gg 1$, $E$ is outgoing as $\chi(x)=1$ for $x\in K$. We have
	\vspace{-2mm}
$$
	-(P-\lambda^4)E_{\chi}(x,\theta,\lambda)= e^{\i\lambda x\cdot \theta}\bigg( \lambda\big(-4\i(\theta\cdot\nabla)(\chi a_2) +\i\theta\cdot\A)\big) - \big( 4\i(\theta\cdot\nabla)(\chi a_3) + 2(\Delta+2(\theta\cdot\nabla)^2)(\chi a_2) +V\big)\vspace{-1mm}
$$
$$
	+\frac{1}{\lambda}\big(- 2(\Delta+2(\theta\cdot\nabla)^2)(\chi a_3) +\i(4(\theta\cdot\nabla)\Delta + \A\cdot\theta)(\chi a_2)\big)
	+\frac{1}{\lambda^2}\big(\i(4(\theta\cdot\nabla)\Delta + \A\cdot\theta)(\chi a_3) + (\Delta^2+\A\cdot\nabla+V)(\chi a_2)\big)\vspace{-1mm}
$$
$$	
	+\frac{1}{\lambda^3}(\Delta^2+\A\cdot\nabla+V)(\chi a_3)\bigg)
$$
	Using the Lie bracket notation and the fact that $\chi \A=\A$ and $\chi V=V$
$$
	-(P-\lambda^4)E_{\chi}(x,\theta,\lambda)= e^{\i\lambda x\cdot\theta}\bigg(-4\i\lambda^3 [\theta\cdot\nabla,\chi]\bold{\tilde{a}} - 2 \lambda^2\big([\Delta+2(\theta\cdot\nabla)^2,\chi]\bold{\tilde{a}}\big) + \i\lambda [4\theta\cdot\nabla\Delta+\A\cdot\theta,\chi]\bold{\tilde{a}}\vspace{-2mm}
$$
$$
	+ [\Delta^2+\A\cdot\nabla+V,\chi]\bold{\tilde{a}} + \frac{\chi}{\lambda}\bigg(-2(\Delta+2(\theta\cdot\nabla)^2)a_3 + \lambda(\Delta^2+\A\cdot\nabla+V)\bold{\tilde{a}} + \i\lambda^2(4\theta\cdot\nabla\Delta + \A\cdot\theta)\bold{\tilde{a}}\bigg) = e^{\i\lambda x\cdot\theta}(E_1 + \frac{\chi}{\lambda}E_2).\vspace{-1mm}
$$
	Now we may apply the resolvent $R(\lambda)$ to both sides of the above equation as the right hand side is compactly supported. 	
	Our goal is to show that $\norm {E_{\chi}}_{L^2} = \mathcal{O}(\lambda^{-4}),$ for $\lambda>0$ large enough. Note that since $\bold{\tilde{a}} = \mathcal{O}(\lambda^{-2}),$ the problematic terms in establishing the estimate for $E_{\chi}$ in the right hand side of above equation will be 
\be{3.3}
	 \lambda^2[\Delta+2(\theta\cdot\nabla)^2,\chi]\bold{\tilde{a}}=\mathcal{O}(1) \quad 
	 \textbf{\normalfont and} \quad \lambda^3[\theta\cdot\nabla,\chi]\bold{\tilde{a}}=\mathcal{O}(\lambda),
\ee 
	as the rest of above terms are of $\mathcal{O}(\lambda^{-1})$ which combined with the resolvent estimates given by Lemma 6.1 gives the desired estimates for $E_{\chi}$. Consider the first term in \r{3.3}. By the resolvent identity 
	\vspace{-1mm}
$$
	R_0(\lambda)-R(\lambda)=R_0(\lambda)(\A\cdot\nabla+V)R(\lambda) \quad \Longrightarrow \quad R(\lambda)=R_0(\lambda)(I-(\A\cdot\nabla+V)R(\lambda)).\vspace{-1mm}
$$
	Therefore,
	\vspace{-1mm}
$$
	R(\lambda)\big[\lambda^2e^{\i\lambda x\cdot\theta}[\Delta+2(\theta\cdot\nabla)^2,\chi]\bold{\tilde{a}}\big]
	=R_0(\lambda)(I-(\A\cdot\nabla+V)R(\lambda))\big[\lambda^2e^{\i\lambda x\cdot\theta}[\Delta+2(\theta\cdot\nabla)^2,\chi]\bold{\tilde{a}}\big]\vspace{-1mm}
$$
$$
	=R_0(\lambda)\big[\lambda^2e^{\i\lambda x\cdot\theta}[\Delta+2(\theta\cdot\nabla)^2,\chi]\bold{\tilde{a}}\big] -R_0(\lambda)(\A\cdot\nabla+V)R(\lambda)\big[\lambda^2e^{\i\lambda x\cdot\theta}[\Delta+2(\theta\cdot\nabla)^2,\chi]\bold{\tilde{a}}\big]
	= \mathcal{I}_1 + \mathcal{I}_2.
$$
	Similarly, for the second term in \r{3.3}, one has
	\vspace{-1mm}
$$
	R(\lambda)\big[\lambda^3 e^{\i\lambda x\cdot\theta}[\theta\cdot\nabla,\chi]\bold{\tilde{a}}\big]
	=R_0(\lambda)(I-(\A\cdot\nabla+V)R(\lambda))\big[\lambda^3 e^{\i\lambda x\cdot\theta}[\theta\cdot\nabla,\chi]\bold{\tilde{a}}\big] \vspace{-1mm}
$$
$$
	=R_0(\lambda)\big[\lambda^3 e^{\i\lambda x\cdot\theta}[\theta\cdot\nabla,\chi]\bold{\tilde{a}}\big] - R_0(\lambda)(\A\cdot\nabla+V)R(\lambda)\big[\lambda^3 e^{\i\lambda x\cdot\theta}[\theta\cdot\nabla,\chi]\bold{\tilde{a}}\big]
	 = \mathcal{J}_1 + \mathcal{J}_2.
$$
	Now we are ready establish estimates for $\mathcal{I}_1$, $\mathcal{I}_2$, $\mathcal{J}_1$, and $\mathcal{J}_2$. Note that the integrands in above free-resolvent operators are compactly supported and therefore, all integrals above are well-defined (see equation \r{2.1}).
	  
	\underline{\textbf{Estimating $\mathcal{I}_2$, $\mathcal{J}_2$.}} By Lemma 6.1,
	\vspace{-2mm}
$$
	\norm {\mathcal{I}_2}_{L^2 _{-\delta}} =
	\norm {R_0(\lambda)(\A\cdot\nabla+V)R(\lambda)\big[\lambda^2e^{\i\lambda x\cdot\theta}[\Delta+2(\theta\cdot\nabla)^2,\chi]\bold{\tilde{a}}\big]}_{L^2 _{-\delta}}\vspace{-2mm}
$$
$$
	\le \frac{C_0}{\lambda} \norm {(\A\cdot\nabla+V)R(\lambda)\big[e^{\i\lambda x\cdot\theta}[\Delta+2(\theta\cdot\nabla)^2,\chi]\bold{\tilde{a}}\big]}_{L^2 _{\delta}}
	\le \frac{C_0}{\lambda}\bigg(\norm {\A\cdot\nabla R(\lambda)\big[e^{\i\lambda x\cdot\theta}[\Delta+2(\theta\cdot\nabla)^2,\chi]\bold{\tilde{a}}\big]}_{L^2 _{\delta}} \vspace{-1mm}
$$
$$
	+\ \norm {VR(\lambda)\big[e^{\i\lambda x\cdot\theta}[\Delta+2(\theta\cdot\nabla)^2,\chi]\bold{\tilde{a}}\big]}_{L^2 _{\delta}}\bigg)
	\le \frac{C_0}{\lambda}(\frac{C_1}{\lambda^4} + \frac{C_2}{\lambda^5}) = \mathcal{O}(\lambda^{-5}),
$$
	where we used Lemma 6.1 and the fact that $\A$, $V$ are compactly supported. Similarly, 
	\vspace{-1mm}
$$
	\norm {\mathcal{J}_2}_{L^2 _{-\delta}} =\norm {R_0(\lambda)(\A\cdot\nabla+V)R(\lambda)\big[\lambda^3 e^{\i\lambda x\cdot\theta}[\theta\cdot\nabla,\chi]\bold{\tilde{a}}\big]}_{L^2 _{-\delta}}
	\le C_0 \norm {(\A\cdot\nabla+V)R(\lambda)\big[e^{\i\lambda x\cdot\theta}[\theta\cdot\nabla,\chi]\bold{\tilde{a}}\big]}_{L^2 _{\delta}}\vspace{-2mm}
$$
$$
	\le C_0 \bigg(\norm {\A\cdot\nabla R(\lambda)\big[e^{\i\lambda x\cdot\theta}[\theta\cdot\nabla,\chi]\bold{\tilde{a}}\big]}_{L^2 _{\delta}} 
	+\ \norm {VR(\lambda)\big[e^{\i\lambda x\cdot\theta}[\theta\cdot\nabla,\chi]\bold{\tilde{a}}\big]}_{L^2 _{\delta}}\bigg)
	\le C_0(\frac{C_1}{\lambda^4} + \frac{C_2}{\lambda^5}) = \mathcal{O}(\lambda^{-4}),\vspace{-1mm}
$$
	\underline{\textbf{Estimating $\mathcal{I}_1$, $\mathcal{J}_1$}}. We mainly follow the idea in \cite{13} to estimate $\mathcal{I}_1$, $\mathcal{J}_1$. Let $\Gamma$ be the set where the derivatives of $\chi$ is supported, and
	\vspace{-2mm} 
$$
	S=\{y=x+t\theta \ | \ x\in B(0,R),\ t\ge0\}.
$$
	Then for any $x\in B(0,R)$ and $y\in \Gamma \cap S$ the kernel $G_0(x,y,\lambda)$, given by \r{2.1}, is smooth as $x\not=y$. We have
	\vspace{-2mm}
$$
	\mathcal{I}_1=R_0(\lambda)(\lambda^2e^{\i\lambda x\cdot\theta}[\Delta+2(\theta\cdot\nabla)^2,\chi]\bold{\tilde{a}})(x)
	=\int \frac{e^{\i\lambda(|x-y|+y\cdot\theta)}-e^{-\lambda|x-y|}e^{\i\lambda y\cdot\theta}}{8\pi|x-y|}[\Delta_{y}+2(\theta\cdot\nabla_{y})^2,\chi_{_y}]\bold{\tilde{a}}(y)\d y \vspace{-1mm}
$$
$$
	= \int K(x,y,\lambda,\theta)[\Delta_{y}+2(\theta\cdot\nabla_{y})^2,\chi_{_y}]\bold{\tilde{a}}(y) \d y\vspace{-1mm}
$$
	Splitting the above integral, we have oscillating integrals with a real phase function $\phi_1(x,y)=|x-y|+y\cdot\theta$ and a complex phase function $\phi_2(x,y)=y\cdot\theta +\i|x-y|$. The contribution of the second phase $\phi_2$ is exponentially small as there is a lower bound of $|x-y|$. In fact, for $\lambda$ large enough, away from the diagonal $\{x=y\}$, the term $e^{-\lambda|x-y|}$ exponentially approaches to zero. Therefore, we can concentrate on the first phase function $\phi_1,$ by employing the stationary phase method. By a simple calculation, one has
	\vspace{-1mm}
$$
	\nabla_y \phi_1 = \frac{y-x}{|y-x|} + \theta \quad \Longrightarrow \quad \theta\cdot\nabla_y \phi_1= \frac{(y-x)\cdot\theta}{|y-x|} + 1 > 1,\vspace{-1mm}
$$
	for any $x\in B(0,\R)$ and $y\in \Gamma \cap S$. Since 
	\vspace{-2mm}
$$	
	e^{\i\lambda\phi_1}= \frac{\nabla_y\phi_1\cdot\nabla_y}{\i\lambda|\nabla_y\phi_1|^2}e^{\i\lambda\phi_1},\vspace{-1mm}
$$ 
	multiple integration by parts yields $
	\norm {\mathcal{I}_1}_{L^2 _{-\delta}(\R^n)} =\mathcal{O}(\lambda^{-4}).$
	Similarly, several integration by parts on
	\vspace{-2mm} 
$$
	\mathcal{J}_1=R_0(\lambda)(\lambda^3 e^{\i\lambda x\cdot\theta}[\theta\cdot\nabla,\chi]\bold{a})(x)
	=\int \lambda K(x,y,\lambda,\theta)[\theta\cdot\nabla_{y},\chi_{_y}]\bold{\tilde{a}}(y) \d y.\vspace{-1mm}
$$
	yields $\norm {\mathcal{J}_1}_{L^2 _{-\delta}(\R^n)} =\mathcal{O}(\lambda^{-4}).$ It remains to establish $\norm {E}_{H^1 _{-\delta}(\R^n)} =\mathcal{O}(\lambda^{-3}).$ Consider the operator $\nabla R(\lambda): L^2 _\delta(\R^n) \rightarrow L^2 _\delta(\R^n)$, and
	\vspace{-2mm}
\be{3.4}
	\nabla R(\lambda)[e^{\i\lambda x\cdot\theta}(E_1 + \frac{\chi}{\lambda}E_2)]. \vspace{-1mm}
\ee	
	Note that since $R(\lambda)$ is the resolvent with constant coefficient, the gradient and the resolvent commute. Therefore, the problematic terms in \r{3.4} will be the same as ones discussed in \r{3.3} and all terms of the form $R_0(\lambda)f$ and $R(\lambda)f$ in \r{3.4} will have the desired $H^1 _\delta$ estimates by Agmon's estimate and the Lemma 6.1. Now we need to revisit the argument for those terms in \r{3.4} with stationary phase similar to $\mathcal{I}_1$ and $\mathcal{J}_1$, where an integration by part argument has been used to establish estimates. We recall that we do not use Agmon's estimates to find $H^1 _\delta$ estimates. Applying the gradient directly to $K(x,y,\lambda,\theta)$, one has 
	\vspace{-1mm}
$$
	K(x,y,\lambda,\theta) = \lambda \tilde{K}(x,y,\lambda,\theta)\vspace{-1mm}
$$	
	 where $\tilde{K}$ is smooth bounded away from diagonal $\{x=y\}$. Now integration by parts argument establishes the desired $H^1 _\delta$ estimates of $\mathcal{O}(\lambda^{-N})$, where $N$ depends on the regularity of magnetic field $\A$ and potential function $V$. This proves the proposition.  
\end{proof}

\section{Near-field Scattering of Biharmonic Solutions}
	In this section we study the near-field scattering problem. Let $u$ be the biharmonic solution of equation \r{1.1} and $B_{R}$ be a ball with radius $R>0$ large enough containing the perturbation $\A$ and $V$ ($\supp \A \cup \supp V \subset B_R$). To study the near-field scattering we only consider $u_{|_{x\cdot\theta=R}}$ as our scattering data and do not study the scattering amplitude to reconstruct the high-frequency asymptotic expansion of the solution. In the following proposition, we demonstrate that all terms up to order $\mathcal{O}(\lambda^{-3})$ contains no additional information.
\begin{proposition}
	Assume that the scattering data $u_{|_{x\cdot\theta=R}}$ is known up to error of order $\mathcal{O}(\lambda^{-3})$. Then\\ 
	i) The scattering data recovers $\curl \A$ and $V-\frac{1}{2}\nabla\cdot\A$.\\
	ii) The scattering data known up to error of order $\mathcal{O}(\lambda^{-3})$ contains no additional information on $\A$ and $V$. 
\end{proposition}	 
\begin{proof}
	i) Let $u$ and $\tilde{u}$ be a pair of biharmonic solutions corresponding to pairs $(\A,V)$ and $(\tilde{\A},\tilde{V})$ such that 
	\vspace{-1mm}
$$	
	{(u-\tilde{u})}_{|_{x\cdot\theta=R}} = \mathcal{O}(\lambda^{-3}). \vspace{-1mm}
$$
	By Proposition 3.1, we have  
	\vspace{-2mm}
$$
	a_2(x) =\frac{1}{4} \int_{-\infty}^0 \theta\cdot\A(x+s\theta)\ ds,
$$ 
	which is the X-ray transform of the magnetic field $\A$ along the lines $x+t\theta$, see \cite{10}. The function $\delta u=u-\tilde{u}$ has near-field data up to error of order $\mathcal{O}(\lambda^{-3})$, hence $\delta a_2\equiv0$ implies that there exists a compactly supported function $\phi$ such that $\delta\A= d\phi$. This shows that the scattering data recovers $\curl \A$. To show the scattering data recovers $V-\frac{1}{2}\nabla\cdot\A$, we recall that 
	\vspace{-1mm}
$$
	4\i(\theta\cdot\nabla) a_3(x)= -2\Delta a_2(x) - 4(\theta\cdot\nabla)^2a_2(x) + V(x).\vspace{-1mm}
$$
	To have an explicit formula for $a_3$ we need to invert the Radon transform as follows: by integrating above equation along the flow and using the fact that $\theta\cdot\nabla=\partial_t$ and $(\theta\cdot\nabla)a_2(x) =\frac{1}{4}\theta\cdot\A(x)$ we have
	\vspace{-1mm}
$$
	a_3(x)= \frac{1}{4\i}[-2\int_{-\infty}^0 \Delta a_2{\scriptstyle(x+s\theta)}\ ds - 4 \int_{-\infty}^0(\theta\cdot\nabla)^2a_2{\scriptstyle(x+s\theta)}\ ds + \int_{-\infty}^0V{\scriptstyle(x+s\theta)}\ ds \vspace{-2mm}	
$$	
$$
	\Longrightarrow \quad \quad a_3{\scriptstyle(x)}= \frac{1}{4\i}\int_{-\infty}^0 V{\scriptstyle(x+s\theta)} -2 \Delta a_2{\scriptstyle(x+s\theta)}\ ds + \frac{\i}{4}\theta\cdot\A{\scriptstyle(x)}.\vspace{-1mm}
$$
	Again since $\delta u=u-\tilde{u}$ has near-field data up to error of order $\mathcal{O}(\lambda^{-3})$ and $\delta\A $ vanishes outside of the support, $\delta a_3\equiv0$ implies that
	\vspace{-2mm}
$$
	\int_{-\infty}^0 2\delta V{\scriptstyle(x+s\theta)} -\Delta \phi{\scriptstyle(x+s\theta)}\ ds =0,\vspace{-1mm}
$$ 	
	as $\delta a_2(x) =\frac{1}{4} \phi(x)$. For $x\cdot\theta=R,$ the standard arguments for inverting the X-ray transform implies that  
	\vspace{-1mm}
$$
	2\delta V = 2(V - \tilde{V}) = \Delta \phi= \nabla \cdot d\phi = \nabla \cdot \delta\A.\vspace{-1mm}
$$	
	This shows that the scattering data up to error of order $\mathcal{O}(\lambda^{-3})$ recovers $V-\frac{1}{2}\nabla\cdot\A$.\\
	\vspace{-2mm} 
	ii) By part i) we know that $\delta \A=d\phi$ and $\delta V = \frac{1}{2}\Delta \phi$. Therefore, point-wise we have
$$
	\left\{\begin{array}{ll}
	\delta a_2(x) =\frac{1}{4} \int_{-\infty}^0 \theta\cdot\nabla\phi{\scriptstyle(x+s\theta)}\ ds = \frac{1}{4} \int_{-\infty}^0 \partial_s\phi{\scriptstyle(x+s\theta)}\ ds  \quad \Longrightarrow \quad {\delta a_2(x) =\frac{1}{4} \phi(x)}\\\delta a_3(x) = \frac{\i}{4}\theta\cdot\nabla\phi(x),
	\end{array}
	\right. \vspace{-1mm}
$$
	which shows that the scattering data contains no additional information on $\A$ and $V$ up to error of $\mathcal{O}(\lambda^{-4})$. In what follows, we calculate $\delta a_4$ and $\delta a_5$, and show that due to non-linearity, the Radon and inverse Fourier transform techniques do not provide any insight in how to show that $\phi=0$.
	\vspace{-1mm}
$$
	\delta a_4(x) =\frac{1}{4\i}[\int_{-\infty}^0-2\Delta \delta a_3{\scriptstyle(x+s\theta)}-4(\theta\cdot\nabla)^2\delta a_3{\scriptstyle(x+s\theta)} + 4\i(\theta\cdot\nabla)\Delta \delta a_2{\scriptstyle(x+s\theta)} + \i\theta\cdot\delta(\A a_2){\scriptstyle(x+s\theta)}\ ds].\vspace{-1mm}
$$
	The nonlinear term $\delta a_4$ can be rewritten as
	\vspace{-2mm} 
$$
	\theta\cdot\delta(\A a_2) = (\theta\cdot\A)\delta a_2 + (\theta\cdot\delta\A) \tilde{a}_2 = \frac{1}{4} (\theta\cdot\A)\phi + (\theta\cdot\nabla\phi)a_2 -\frac{1}{4}(\theta\cdot\nabla\phi)\phi \vspace{-2mm}
$$
$$
	= \frac{1}{4} (\theta\cdot\A)\phi + (\theta\cdot\nabla)(\phi a_2) - (\theta\cdot\nabla a_2)\phi -\frac{1}{4}(\theta\cdot\nabla\phi)\phi =  (\theta\cdot\nabla)(\phi a_2) -\frac{1}{4}(\theta\cdot\nabla\phi)\phi,
$$
	where we used the equations for $\delta a_2, \delta a_3$, $\delta \A$, and the fact that $\frac{1}{4} (\theta\cdot\A) = (\theta\cdot\nabla a_2)$. By a simple calculation
	\vspace{-2mm} 
$$	
	\delta a_4(x) = \frac{1}{8} [\Delta\phi(x) - 2(\theta\cdot\nabla)^2\phi(x) - \frac{1}{4}\phi^2(x) + 2\phi(x)a_2(x)].	
$$
\end{proof}
\vspace{-7mm}		
\section{Proof of the Main Results}
	In this section we prove the main result. We first derive the asymptotic expansion of scattering amplitude which is known as Born approximation, see also \cite{16} .
\begin{theorem}
	Let $V , \A \in C^k (\R^3)$ for $k$ large enough. Then
	\vspace{-1mm}
$$
	a(\omega,\theta,\lambda)=\i\lambda\theta\cdot\hat{A}(\lambda(\omega-\theta)) + \hat{V}(\lambda(\omega-\theta)) + \mathcal{O}(\lambda^{-1})
	\quad \text{\rm as \quad $\lambda \rightarrow \infty$},\vspace{-1mm}
$$
	with the remainder uniform in $\theta$, $\omega$.	
\end{theorem}
\begin{proof}
	Let $u=e^{\i\lambda x\cdot \theta}+ e^{\i\lambda x\cdot \theta}(\bold{a}-1)$ be the biharmonic solution given by Proposition 3.1. Plugging $u$ into \r{2.2}, we have
	\vspace{-2mm}
%$$
%	a(\omega,\theta,\lambda)=\int e^{-\i\lambda\omega\cdot y}\big(\A(y)\cdot\nabla + V(y)\big)(e^{\i\lambda y\cdot \theta}) \d y + \int e^{-\i\lambda\omega\cdot y}\big(\A(y)\cdot\nabla + V(y)\big)(e^{\i\lambda y\cdot \theta} (\bold{a}-1)(y,\theta,\lambda)) \d y
%$$
%\vspace{-2mm}
%$$
%	a(\omega,\theta,\lambda)= \int e^{-\i\lambda(\omega-\theta)\cdot y}\big(\i\lambda\theta\cdot\A(y) + V(y)\big) \d y + \int e^{-\i\lambda\omega\cdot y}\big(\A(y)\cdot\nabla + V(y)\big)(e^{\i\lambda y\cdot \theta}(\bold{a}-1)) \d y
%$$
$$
	a(\omega,\theta,\lambda)= \int e^{-\i\lambda(\omega-\theta)\cdot y}\big(\i\lambda\theta\cdot\A{\scriptstyle(y)} + V{\scriptstyle(y)}\big) \d y + \int e^{-\i\lambda(\omega-\theta)\cdot y}\big(\i\lambda\theta\cdot\A{\scriptstyle(y)} + \A{\scriptstyle(y)}\cdot\nabla+ V{\scriptstyle(y)}\big)(\bold{a}-1){\scriptstyle(y)} \d y\vspace{-1mm}
$$
\be{5.1}
	=\i\lambda\theta\cdot\hat{A}(\lambda(\omega-\theta)) + \hat{V}(\lambda(\omega-\theta)) + R(\omega,\theta,\lambda).
\ee	
	Now by Proposition 3.1 we have $\norm{\bold{a}-1}_{L^2(\R^2)},\ \norm{\nabla(\bold{a}-1)}_{L^2(\R^2)}\le C\lambda^{-2}$
	for some constant $C$. Therefore $R=\mathcal{O}(\lambda^{-1})$ which completes the proof.
\end{proof}
\vspace{-3mm}
\begin{remark}
	\rm 
	Given $\A, V \in C^{\infty}_0$, for $\lambda$ large enough, the first two terms on the right hand side of above amplitude $a$ decay faster that the remainder, if $\omega \not= \theta$ are fixed. In other words, $\frac{a(\omega,\theta,\lambda)}{\lambda}$ is bounded for $\A$ and $V$ regular enough. Therefore, $\supp_{\substack{\omega,\theta,\lambda\\\scriptstyle 0<\lambda_0\le\lambda}} \big|\frac{1}{\lambda}a{\scriptstyle  (\omega,\theta,\lambda)}\big|$ is a well-defined norm for a fixed $\lambda_0>0.$ 
\end{remark}

\begin{theorem}
	Let $\A, V \in C^k(\R^3)$ for $k$ large enough and $\theta\in S^2$ be fixed. Then for any $\xi\not=0$ with $\xi\cdot\theta=0$, the scattering amplitude $a(\omega,\theta,\lambda)$
	uniquely determines $\theta\cdot\hat{\A}{\scriptstyle(\xi)}$ and $-\i\xi\cdot\hat{\A}{\scriptstyle(\xi)} + 2\hat{V}{\scriptstyle(\xi)}$.
\end{theorem}
	
\begin{proof}
	Let $\theta \in S^2$ be a fixed unit vector. For a fixed $\xi\not=0$ with $\xi\ \bot\ \theta$ we show that one can construct a sequence $\{(\omega_{\mu},\tilde{\omega}_{\mu},\lambda_{\mu})\}_{\mu}$ such that
	\vspace{-2mm}
$$
	\xi=\lambda_{\mu}(\omega_{\mu}-\tilde{\omega}_{\mu}).	
$$
	Let $\omega,\tilde{\omega}$ be two vectors that are symmetric with respect to the fixed unit vector $\theta$. Choose the parameter $\mu\in \R$ small enough (see \cite{11}) such that
	\vspace{-2mm}
$$
	\omega^+ = \omega^+_{\mu}= \theta \cos\mu + \frac{\xi}{|\xi|}\sin \mu, \quad \omega^-=\omega^-_{\mu}=\theta \cos\mu - \frac{\xi}{|\xi|}\sin \mu. \vspace{-2mm}
$$
	Clearly for $\xi\not=0$, 
	\vspace{-2mm}
$$
	|\omega^+|=|\omega^-|=1, \quad \omega^+-\omega^-= 2\frac{\xi}{|\xi|}\sin \mu, \quad \mu = \sin^{-1}(2\frac{\xi}{|\xi|}(\omega^+-\omega^-)).\vspace{-1mm}
$$
	Setting $\lambda{\scriptstyle(\mu)} = \frac{|\xi|}{2\sin\mu}$ yields $ \xi=\lambda(\omega^+-\omega^-).$ Note that $
	\omega^+, \omega^-\rightarrow\theta$ and $\lambda \rightarrow \infty$ as $\mu\rightarrow 0.$  
	By Theorem 5.1,
$$
	\left\{\begin{array}{ll}
	a^+=a{\scriptstyle(\omega^+,\omega^-,\lambda)}=\i\lambda\omega^-\cdot\hat{\A}{\scriptstyle (\xi)} + \hat{V}{\scriptstyle (\xi)} + \mathcal{O}(\lambda^{-1})\\
	a^-=a{\scriptstyle(-\omega^-,-\omega^+,\lambda)}=-\i\lambda\omega^+\cdot\hat{\A}{\scriptstyle (\xi)} + \hat{V}{\scriptstyle (\xi)} + \mathcal{O}(\lambda^{-1}). 
	\end{array}
	\right. 
$$
	with remainders uniform in $\omega^+,\omega^-$. Therefore
	\vspace{-2mm}
\be{5.2}
    \left\{\begin{array}{ll}
    a^+ - a^-= 2\i\lambda\cos \mu \ \theta\cdot\hat{\A}{\scriptstyle (\xi)} + \mathcal{O}(\lambda^{-1})\\a^+ + a^-= -\i\xi\cdot\hat{\A}{\scriptstyle (\xi)} + 2\hat{V}{\scriptstyle (\xi)} + \mathcal{O}(\lambda^{-1}). 
    \end{array}
    \right. 
\ee
	The analogous formulae for the Schr\"{o}dinger operator are presented in \cite{9,13}. The first equation above implies that $\theta\cdot\hat{\A}{\scriptstyle (\xi)}$ can be recovered and the second equation implies that one can reconstruct $-\i\xi\cdot\hat{\A}{\scriptstyle (\xi)} + 2\hat{V}{\scriptstyle (\xi)}$ for $\theta\ \bot\ \xi$ which completes the proof. 
\end{proof}
	We now are ready to prove our main result.	
\begin{proof}[\textbf{Proof of Theorem 1.1.}]	
	By Theorem 5.2. we know that $\theta\cdot\hat{\A}{\scriptstyle (\xi)}$ can be recovered for any $0\not=\xi\ \bot\ \theta$. For a non-zero vector $\alpha\in\R^3$, set $\theta=\frac{\alpha\times\xi}{|\alpha\times\xi|}$. By Theorem 5.2. 
	\vspace{-2mm} 
$$
	(\alpha\times\xi)\cdot\hat{\A}{\scriptstyle (\xi)}=\alpha\cdot(\xi\times\hat{\A}{\scriptstyle (\xi)}),\vspace{-1mm}
$$
	is known as clearly $(\alpha\times\xi)\ \bot\ \xi$. Note that the r.h.s of above equation is the Fourier transform of the $\curl\A$ projected on an arbitrary non-zero vector $\alpha\not=0$. Therefore, one can recover 
	\vspace{-2mm}
$$
	\xi\times\hat{\A}(\xi) = \frac{1}{\i}\mathcal{F}(\curl \A)
$$
	for $\xi\ \bot\ \theta$. Since $\theta$ is arbitrary, $\mathcal{F}(\curl \A)$ can be recovered everywhere. Taking the inverse Fourier transform yields $\curl \A$ can be recovered. On the other hand by Theorem 5.2, we know that $-\i\xi\cdot\hat{\A} + 2\hat{V}$ can be recovered. Since
	\vspace{-2mm} 
$$
	-\i\xi\cdot\hat{\A} + 2\hat{V} = \mathcal{F}(-\nabla\cdot\A + 2V) 
$$
	one can recover $V - \frac{1}{2}\nabla\cdot\A$ which completes the proof of the main theorem.
\end{proof}

	We next prove stability estimate results.
\begin{proposition}
	For $j=1,2,$ let $a_j$ be the amplitude with corresponding pair of a priori bounded magnetic and potentials fields $(\A_j,V_j),$ i.e. $\parallel A_j \parallel_{C^k} < C_0$ and $\parallel V_j \parallel_{C^k} < C_0$ for some constant $C_0>0$ and $k>0$. 
%	There exists $K>0$ large enough such that if
%$$
%	\parallel \hat{\A}_1-\hat{\A}_2\parallel_{C^K}<\delta \ll 1, \quad \parallel\hat{V_1}-\hat{V_2}\parallel_{C^K} <\delta \ll 1,
%$$
	Then there exists $C > 0$ and $\lambda_0$ depending a priori on $C_0$ such that for $\epsilon = \sup_{\substack{\omega,\theta,\lambda\\\scriptstyle 0<\lambda_0\le\lambda}} \big|\frac{1}{\lambda}(a_1-a_2){\scriptstyle (\omega,\theta,\lambda)}\big|$ small enough, the following stability estimates hold 
$$
	\sup_{\xi}|\frac{(\widehat{\curl \A}_1-\widehat{\curl \A}_2)(\xi)}{\langle \xi \rangle}|<\epsilon, \quad \quad \parallel\widehat{(V_1-V_2)} -\frac{1}{2}\i\xi\cdot\widehat{(\A_1-\A_2)}\parallel_{L^{\infty}} <C\epsilon^{-\frac{1}{2}}.\vspace{-1mm}
$$	
\end{proposition}
\begin{proof}
	Let $u_j= e^{\i\lambda x\cdot \theta}\bold{a}_j$, $j=1,2,$ be the biharmonic solution given by the Proposition 3.1 with corresponding amplitude $a_j$ and the pair $(\A_j,V_j)$. Similar to Theorem 5.2, for a fixed $\xi\not=0$ with $\xi\ \bot\ \theta$ and $\mu \ll \delta$ (i.e small enough),  one can choose $\lambda_{\mu} = \frac{|\xi|}{2\sin\mu }$, $\omega^+ _{\mu},$ and $\omega^- _{\mu}$, so that $\xi=\lambda_{\mu}(\omega^+ _{\mu}-\omega^- _{\mu})$. Note that for small enough $\mu,$ one has $\frac{|\xi|}{2\lambda_{\mu}} = \sin\mu \le \mu\ll \delta$ and therefore $\frac{|\xi|}{2\delta} \le \lambda_{\mu}.$  Set $C_0 = \frac{1}{2\delta}$ and $\lambda_0=\frac{|\xi|}{2\delta}$. Assume now that the pair $(\A_j,V_j)$ satisfies the priori assumption in the theorem. For a fixed and large $k,$ by equation \r{5.2} we have 
\be{5.3}
	\left\{\begin{array}{ll}
	a^+ _j - a^- _j= 2\i\lambda\cos \mu \ \theta\cdot\hat{\A}_j {(\xi)} + (R^+ _j - R^- _j)\\
	a^+ _j + a^- _j= -\i\xi\cdot\hat{\A}_j{(\xi)} + 2\hat{V}_j{(\xi)} + (R^+ _j + R^- _j), 
	\end{array}
	\right.  \quad \text{ with \quad $|R^{\pm} _j|\le \frac{C}{\lambda},\quad j=1,2,$}
	\vspace{-1mm}
\ee 
%	for $j=1,2$
%$$
%	a_j(\omega,\theta,\lambda)=\i\lambda\theta\cdot\hat{A}_j(\lambda(\omega-\theta)) + \hat{V}_j(\lambda(\omega-\theta)) +  \int e^{-\i\lambda(\omega-\theta)\cdot y}\big(\i\lambda\theta\cdot\A(y) + \A(y)\cdot\nabla+ V(y)\big)\bold{b}\ \d y.
%$$
%	Therefore,
%$$
%	(a_1-a_2)(\omega,\theta,\lambda)=\i\lambda\theta\cdot(\hat{A}_1-\hat{A}_2)(\lambda(\omega-\theta)) + (\hat{V}_1-\hat{V}_2)(\lambda(\omega-\theta)) + R
%$$
	 where $C$ depends on the priori upper bound $C_0$.
	%By Proposition 3.1, we have $\norm{\bold{a}_j-1}_{L^2(\R^2)},\ \norm{\nabla(\bold{a}_j-1)}_{L^2(\R^2)}\le C\lambda^{-2}$ for some constant $C$.
	Using the first equation in \r{5.3} and setting $R^{\pm}= (R_1-R_2)^{\pm},$ we have
	\vspace{-1mm}
$$
	\theta\cdot\delta\hat{A}{(\xi)} = \frac{1}{2\i\lambda_{\mu}\cos\mu}\big[(a_1 - a_2)^+ - (a_1-a_2)^- - (R^+ - R^-)\big].
$$	
	Note that for $\mu$ small enough, one has $\mu \approx \frac{|\xi|}{2\lambda_{\mu}}$ and $\frac{1}{\cos\mu} = 1 + \frac{1}{2}\mu^2 + \mathcal{O}(\mu^4)\le 1 + \mu^2$. Therefore, 
	\vspace{-2mm}
$$
	|\theta\cdot\delta\hat{A}{(\xi)}| \le (1 + \frac{|\xi|^2}{\lambda^2}) \big[|\frac{1}{2\lambda_{\mu}}(a_1-a_2)^+| + |\frac{1}{2\lambda_{\mu}}(a_1-a_2)^-| + \frac{1}{2\lambda _{\mu}}|R^+ - R^-|\big].
	\vspace{-1mm}
$$
	To establish an estimate for $\theta\cdot\delta\hat{A}{(\xi)}$ for any fixed $\xi,$ note that $\epsilon = \sup_{\substack{\omega,\theta,\lambda\\ 0<\lambda_0\le\lambda}} \big|\frac{1}{\lambda}(a_1-a_2)\big|$. Therefore,
	\vspace{-3mm}   
$$
	|\theta\cdot\delta\hat{A}{(\xi)}| \le (1 + \frac{|\xi|^2}{\lambda^2 _{\mu}}) (\epsilon + \frac{C}{\lambda^2 _{\mu}}) \rightarrow \epsilon \quad \quad  \text{as \quad\quad  $\mu \rightarrow 0$.}
	\vspace{-2mm} 
$$
%$$	 
%	\frac{1}{2}\parallel \frac{1}{\lambda_{\mu}}(a_1-a_2)^+\parallel_{C^0} +\frac{1}{2} \parallel \frac{1}{\lambda_{\mu}}(a_1-a_2)^-\parallel_{C^0} + \frac{C}{\lambda^2 _{\mu}}\le \frac{1}{\lambda_{\mu}}\big(\epsilon + \frac{C}{\lambda_{\mu}}\big). 
%$$	
	%Note that since $\A$ and $V$ are supported in a fixed ball in $C^k,$ Fourier transforms of the magnetic and potential fields $\A_j$ and $V_j$ are uniformly bounded and therefore the $L^\infty$-norm is well-defined. 
	Similar to the proof of Theorem 1.1, set $\theta=\frac{\alpha\times\xi}{|\alpha\times\xi|}$ for any $\alpha \in \R^3.$ Since $|\alpha\times\xi|\le c_{\alpha}(1+|\xi|)=  c_{\alpha}\langle \xi \rangle$, one has
	\vspace{-2mm} 
$$
	|\frac{\alpha.(\xi\times\delta\hat{A}{(\xi)})}{\langle \xi \rangle}| \le |\frac{(\alpha\times\xi)\cdot\delta\hat{A}{(\xi)}}{|\alpha\times\xi|}| \le \epsilon.
$$ 		
	Now $\alpha$ is arbitrary, so one can stably recover all the components of the $\curl \A$ with the following estimates
	\vspace{-1mm}
$$	
	\sup_{\xi} |\frac{\widehat{\curl(\delta\A)}(\xi)}{\langle \xi \rangle}| < \epsilon. 
$$

	To establish an estimate for $(2\delta\hat{V}-\i\xi\cdot\delta\hat{\A})(\xi)$, we use the second equation of \r{5.3}. We have
%$$
%	(2\delta\hat{V}- \i\xi\cdot\delta\hat{\A}){(\xi)} =  (a_1 - a_2)^+ + (a_1 - a_2)^- - (R^+ + R^-).\vspace{-1mm}
%$$
$$
	|(2\delta\hat{V}- \i\xi\cdot\delta\hat{\A}){(\xi)}| \le |(a_1 - a_2)^+ | + |(a_1 - a_2)^-|+ |R^+ + R^-| \le 2\epsilon\lambda_{\mu} + \frac{C}{\lambda_{\mu}}. 
	\vspace{-2mm}
$$
%	Taking the $L^\infty$-norm, we have
%$$
%	|(2\delta\hat{V}- \i\xi\cdot\delta\hat{\A}){(\xi)}|\le \parallel (a_1 - a_2)^+\parallel_{C^0} +\parallel (a_1 - a_2)^-\parallel_{C^0} + \frac{C}{\lambda_{\mu}} \vspace{-1mm}
%$$
	The r.h.s of above is minimized when $\lambda_{\mu} = \epsilon^{-\frac{1}{2}}.$ Therefore 
$$
	|(2\delta\hat{V}- \i\xi\cdot\delta\hat{\A}){(\xi)}| \le C\epsilon^{\frac{1}{2}}. \vspace{-1mm}
$$	
\begin{remark}
	\normalfont One can establish different norms for $\A$ using different norms of $a$.
\end{remark}

%	 Since for a fixed $\delta \ll 1$ there exists a priori bounds for the sequence of magnetic and potential fields $\A$ and $V$,	
%$$
%	\parallel a_1 - a_2 \parallel_{C^K} \ \le \lambda\parallel \hat{\A}_1-\hat{\A}_2\parallel_{C^K} + \parallel\hat{V_1}-\hat{V_2}\parallel_{C^K} + \mathcal{O}(\lambda^{-1}) \ \le C \delta(\lambda+1+\lambda^{-1}) + \mathcal{O}(\lambda^{-2}).
%$$
%	The right hand side of above is minimized when $\delta = \frac{1}{\lambda^2}$. Therefore
%$$
%	\parallel a_1 - a_2 \parallel_{C^K} \ \le C \delta^{\frac{1}{2}}.
%$$	
%	Now assume that $\parallel a_1 - a_2 \parallel_{C^K} <\epsilon$. Clearly,
%$$
%	\parallel \hat{\A}_1-\hat{\A}_2\parallel_{C^K} \le \parallel a_1 - a_2 \parallel_{C^K}\le C \delta(\lambda+1), \ \parallel\hat{V_1}-\hat{V_2}\parallel_{C^K} \le \parallel a_1 - a_2 \parallel_{C^K}\le C \delta(\lambda+1).
%$$	
	
\end{proof}	
  
\vspace{-8mm}	
\section{Appendix}
	
\begin{definition}{Weighted $L^p$ Spaces.}
	We say the function $f$ belongs to $L^p _{\delta}(\R^n)$ if and only if
	\vspace{-2mm}
$$
	\norm f_{L^p _{\delta}(\R^n)} := (\int_{\R^n} (1+|x|)^{\delta p}|f(x)|^p \d x)^{\frac{1}{p}} < \infty. \vspace{-1mm}
$$
\end{definition}
\begin{definition}{Sobolev Spaces.}
	We say the function $f$ belongs to $W^1_{p,\delta}(\R^n)$ if and only if $f$ and $\nabla f$ belong to the weighted Lebesgue space $L^p _{\delta}(\R^n)$.
	If $p=2$, we denote $W^s_{2,\delta}(\R^n) := H^s _{\delta}(\R^n)$. Note that $H^0 _{\delta}(\R^n)=L^2 _{\delta}(\R^n).$
\end{definition}

The following lemma and corollary are direct results from the Agmon's estimate. For the proof we refer the reader to [Lemma 4.1 \& Theorem 4.2, \cite{15}].
\begin{lemma}{i) Agmon's estimate.}
	The operator $R_0(\lambda)$ maps from $L^2 _{\delta} (\R^n)$ to $H^2 _{-\delta}(\R^n)$, with following estimates
	\vspace{-2mm}
$$
	\norm {R_0(\lambda)f}_{H^k _{-\delta}(\R^n)} \le \frac{C_0}{\lambda^{3-k}} \norm f_{L^2 _{\delta}(\R^n)}, \quad k=0,1,2,\quad \text{\rm and \quad $\delta > \frac{1}{2}$.}
$$ 

	ii) Moreover, we have the following estimate for $R(\lambda)$
	\vspace{-2mm}
$$
	\norm {R(\lambda)f}_{H^k _{-\delta}(\R^n)} \le \frac{C_0}{\lambda^{3-k}} \norm f_{L^2 _{\delta}(\R^n)}, \quad k=0,1,2.
$$ 
\end{lemma}
\begin{proof}
	For the proof see ~[Lemmas 4.1 and 6.1, \cite{15}].
\end{proof}

\end{document}